\documentclass{article}

\usepackage{amsfonts,latexsym,amssymb,amsmath,amsthm,epsfig,amscd,comment,url}

\usepackage[margin=1.2in]{geometry}

\newtheorem{theorem}{Theorem}
\newtheorem{lemma}[theorem]{Lemma}

\newtheorem{claim}[theorem]{Claim}
\newtheorem{prop}[theorem]{Proposition}

\newtheorem{remark}[theorem]{Remark}
\newtheorem{conjecture}[theorem]{Conjecture}
\newtheorem*{rep@theorem}{\rep@title}
\newcommand{\newreptheorem}[2]{%
\newenvironment{rep#1}[1]{%
 \def\rep@title{#2 \ref{##1} (restatement)}%
 \begin{rep@theorem}}%
 {\end{rep@theorem}}}
\makeatother
\newreptheorem{prop}{Proposition}

\newcommand{\suppress}[1]{}

\DeclareMathOperator{\Sen}{Sen}

\newcommand{\cA}{{\cal A}}
\newcommand{\cB}{{\cal B}}
\newcommand{\cN}{{\cal N}}
\newcommand{\cS}{{\cal S}}
\newcommand{\cT}{{\cal T}}

\newcommand{\SB}{{\overline{\cS}}}
\newcommand{\rmx}{{r_{\max}}}

\newcommand{\ra}{\rightarrow}

\newcommand{\ep}{\varepsilon}

\newcommand{\II}{{\mathbb{I}}}

\newcommand{\RR}{{\mathbb{R}}}

\newcommand{\ZZ}{{\mathbb{Z}}}

\newcommand{\be}{\begin{equation}}
\newcommand{\ee}{\end{equation}}
\newcommand{\bea}{\begin{eqnarray}}
\newcommand{\eea}{\end{eqnarray}}
\newcommand{\bean}{\begin{eqnarray*}}
\newcommand{\eean}{\end{eqnarray*}}
\def\ba#1\ea{\begin{align}#1\end{align}}
\def\ban#1\ean{\begin{align*}#1\end{align*}}

\newcommand{\eq}[1]{(\ref{eq:#1})}
\newcommand{\sect}[1]{Sec.~\ref{sec:#1}}
\newcommand{\thmref}[1]{Theorem~\ref{thm:#1}}
\newcommand{\lemref}[1]{Lemma~\ref{lem:#1}}

\newcommand{\propref}[1]{Proposition~\ref{prop:#1}}

\newcommand{\MC}{{\sc Maximum Cut}}

\newcommand{\SC}{{\sc Sparsest Cut}}

\begin{document}
\clubpenalty=10000
\widowpenalty = 10000

\title{
Dimension-free {\huge L$_2$} Maximal Inequality for Spherical Means in the Hypercube}

\author{Aram W. Harrow\footnote{Massachusetts Institute of Technology} \and Alexandra
  Kolla\footnote{University of lllinois, Urbana-Champaign}
  \and Leonard J. Schulman\footnote{California Institute of Technology}}
\maketitle

\begin{abstract}
We establish the maximal inequality claimed in the title.
 In combinatorial terms this has the implication that for sufficiently small $\ep>0$, for all $n$, any marking of an $\ep$ fraction of the vertices of the $n$-dimensional hypercube necessarily leaves a vertex $x$ such that marked vertices are a minority of every sphere centered at $x$.
\end{abstract}

\section{Introduction}\label{sec:intro}
Let $\II^n$ be the $n$-dimensional hypercube:
the set $\{0,1\}^n$ equipped with Hamming metric
$d(x,y)=|\{i:x_i\neq y_i\}|$.
Let $V=\RR^{\II^n}$ be the vector space of real-valued
functions on the hypercube.   For $x \in \II^n$, let $\pi_x$ denote
the {\em evaluation map}\footnote{This notation choice is because our
  paper is replete with operators acting on functions, and the associative composition $\pi_x A f$ is preferable to the cumbersome $(Af)(x)$.}
 from $V\ra \RR$ defined by $\pi_x f = f(x)$,
for $f\in V$.
If $\cA \subseteq
\textrm{Hom}(V,V)$, (the vector space consisting of all linear mappings from V into itself)  the \textit{maximal operator} $M_\cA:V
\to V$ is the sublinear operator in which $M_\cA f$ is defined by
\be \pi_x M_\cA f  = \sup_{A \in \cA} \pi_x Af \ee

Of interest is the family $\cS=\{S_k\}_{k=0}^n$ of
\textit{spherical means}, the stochastic linear operators $
S_k:V  \to  V $ given by
\[ \pi_x S_k f = \sum_{\{y: d(x,y)=k\}} \pi_y f / \binom{n}{k}
\]
Applying $M$, we have the \textit{spherical maximal operator}
$M_{\cS}:V  \to V$ defined by
\be \pi_x
M_{\cS}f  = \max_{0 \leq k \leq n} \pi_x S_k f \ee
The result of this paper is the following dimension-free bound.
\begin{theorem} \label{thm:main} There is a constant $A_\II$
  such that for all $n$, $\|M_{\cS}\|_{2 \to 2}<A_\II$.
 \end{theorem}
Equivalently, for all $n$ and $f$, $\|M_{\cS}f\|_2 \leq A_\II \|f\|_2$.

Maximal inequalities for various function spaces such as
$L_p(\RR^n)$, and with stochastic linear operators defined
by various distributions such as uniform on spheres (as
above), uniform on balls, or according to the distribution
of a random walk of specified length (ergodic averaging),
have been extensively studied; see~\cite{NT10} for a good
review.  Most previous work does not explicitly consider
finite or discrete metric spaces; however, see~\cite{NS94} for a maximal inequality on the free group on
finitely many generators.

One may ask whether the hypercube bound should follow from
known results for larger spaces. The hypercube of dimension greater than $1$ does not
embed isometrically in Euclidean space of any
dimension~\cite{Enflo69,LinialM00}, so inequalities for
Euclidean spaces do not seem to be a useful starting
point. The hypercube does embed isometrically in $\RR^n$
with the $L_1$ norm, but there is no maximal inequality for
this metric. To see this, still in the context of discrete
metric spaces, consider the space $\ZZ^2$ with the $L_1$
distance. Fixing any nonnegative integer $N$ let $f$ be the
indicator function of  $\{x:\sum x_i=0,\sum |x_i|\leq
2N\}$. Then $\|f\|_2^2=2N+1$ while $\|M_{\cS}f\|_2^2 \in
\Omega(N^2)$. A similar gap (between $O(N^{n-1})$ and
$\Omega(N^n)$) occurs in any fixed dimension $n$, because
there exists a set of size $O(N^{n-1})$ constituting a
positive fraction of $\Omega(N^n)$ $L_1$-spheres,
necessarily of many radii. It is therefore not the $L_1$
metric structure of the hypercube which makes a maximal
inequality possible, but, essentially, its bounded side-length.

\subsection{Combinatorial interpretation}
In the special case that $f$ is the indicator function of a
set of vertices $F$ in $\II^n$, \thmref{main} has the
following consequence: For nonnegative $\ep$ less than some
$\ep_0$ and for all $n$, if $|F|<\ep 2^n$ then there exists
$x \in \II^n$ such that in every sphere about $x$, the
fraction of points which lie in $F$ is $O(\sqrt\ep)$.

The aspect of interest is that this holds for every sphere
about $x$. The analogous claim for a fixed radius is a
trivial application of the Markov inequality;
by a union bound the same holds for any constant number of
radii. Avoiding the union bound over all $n+1$ radii is the essence of the
maximal inequality.

The combinatorial interpretation also has an edge
version. Let $F'$ be a set of edges in $\II^n$. Let the
distance from a point to an edge be the distance to the
closer point on that edge. \thmref{main} has
the following consequence: For nonnegative $\ep$ less than
some $\ep_1$ and for all $n$, if $|F'|<\ep n 2^n$ then there
exists $x \in \II^n$ such that in every sphere about $x$,
the fraction of edges which lie in $F'$ is $O(\sqrt\ep)$.
(Define a function $f$ on vertices by $f(y)=$ the fraction
of edges adjacent to $y$ that lie in $F'$. Note that
$\|f\|_2 \in O(\sqrt{\ep})$. Apply \thmref{main} to $f$. For
the desired conclusion observe that in the sphere of edges
of distance $k$ from $x$, the fraction of edges lying in
$F'$ is bounded for $k \leq n/2$ by $2 \pi_xS_kf$, and for $k
> n/2$ by $2\pi_xS_{k+1}f$.)

\subsection{Maximal inequalities and Unique Games on the Hypercube}

In this section we discuss one of the motivations for the present work, although our main result does not directly yield progress on the question.

Khot's \textit{Unique Games Conjecture} (UGC)~\cite{Kho02} has, for a decade, been the focus of great attention.
Either resolution of the conjecture will have implications for the hardness of approximating NP-hard problems. Falsification, in particular, is likely to provide powerful new algorithmic techniques. On the other hand verification of UGC will imply that improving the currently best known approximation algorithms for
many important computational problems such as {\sc Min-2Sat-Deletion}~\cite{Kho02}, {\sc Vertex Cover}~\cite{KhotRegev}, {\MC}~\cite{KKMO} and
non-uniform {\SC}~\cite{CKKRS,KV}, is NP-hard. In addition, in recent years, UGC
has also proved to be intimately connected to the limitations of
semidefinite programming (SDP). Making this connection precise, the authors in~\cite{Aus07} and~\cite{Rag08} show that if UGC is true, then for every constraint satisfaction
problem (CSP) the best approximation ratio is given by a
certain simple SDP.  While the UGC question, in short, is arguably among the most important challenges in computational complexity, the current evidence in either direction is not strong, and there is no consensus regarding the likely answer.

A Unique Game instance is specified by an undirected constraint graph $G = (V,E)$, an integer $k$ which is the alphabet size,
a set of variables $\{x_u\}_{u \in V}$, one for each vertex $u$, and
a set of permutations (constraints) $\pi_{uv}: [k] \rightarrow [k]$,
one for each $(u,v)$ s.t.\ $\{u,v\}\in E$, with $\pi_{uv} = (\pi_{vu})^{-1}$.
An assignment of values in $[k]$ to the variables is said to
satisfy the constraint on the edge $\{u,v\}$ if $\pi_{uv}(x_u) = x_v$.
The optimization problem is to assign a value in $[k]$ to each variable $x_u$ so as to maximize the number of
satisfied constraints.

Khot~\cite{Kho02} conjectured that it is NP-hard
to distinguish between the cases when almost all, or very few, of the constraints of a Unique
Game are satisfiable:

\begin{conjecture}(UGC)
For any constants $\epsilon, \delta > 0$, there is a $k(\epsilon,\delta)$ such that for any $k > k(\epsilon,
\delta)$, it is NP-hard to distinguish between instances of Unique
Games with alphabet size $k$ where at least $1-\epsilon$ fraction of
constraints are satisfiable and those where at most $\delta$
fraction of constraints are satisfiable.
\end{conjecture}

 Numerous works have focused on designing approximation algorithms for Unique Games. While earlier
 papers~\cite{Kho02,Trevisan,GT,CMM1,CMM2} presented efficient algorithms for any instance that might have very poor approximation guarantees when run on the worst case input, recent works have focused on ruling out hardness for large classes of instances. Such instances include expanders~\cite{AKKTSV,MM}, local expanders~\cite{AIMS,RS10}, and more generally, graphs with a few large eigenvalues~\cite{K}. In~\cite{KMM}, hardness of random and semi-random distributions over instances is also ruled out. We note that in~\cite{ABS}, a sub-exponential time algorithm for general instances is given.\\

  This recent line of work can be seen as a strategy to disprove UGC by ruling out hardness of instances largely based on their spectral profile, and more specifically the number of ``large enough" eigenvalues (small set expansion) of the instance's constraint graph. Following this strategy, the question to ask next is what type of graphs are the ones on which all those techniques fail. The hypercube is typical of such graphs, with its spectrum lying in a ``temperate zone" in terms of expansion. This property, along with the high symmetry of the hypercube, makes it a natural next frontier toward disproving UGC.\\

It is typical when dealing with a $1-\epsilon$ satisfiable instance to
think of it as originating from a completely satisfiable instance,
where a malicious adversary picks an $\epsilon$ fraction of edges in the constraint graph and ``spoils'' them, by modifying their corresponding constraints. Some likely algorithms for Unique Games on the hypercube are seed-and-propagate (exhaustively range over assignments for a small subset of nodes and propagate the assignment from those nodes according to the constraints on the edges, using some fixed conflict resolution strategy), or ``local SDP''. In either case the performance of these algorithms depends on whether the adversary can ``rip up'' the graph enough so that the algorithm has difficulty propagating the seeded values, or otherwise combining local solutions into larger regions---even though in a pure expansion sense the graph is not \textit{so} easy to rip up that the algorithm could afford to
work just in local patches of the graph and throw away all the edges between patches. A good strategy to show that the adversary can win, is to show that the adversary can select an $\ep$ fraction of the edges of the graph in such a way
that around every seed vertex $x$ there is a sphere in which a majority of the of the edges have been selected by the adversary.

\thmref{main} shows that, to the contrary, no adversary has this power.

It is a challenging problem to analyze the performance of seed-and-propagate algorithms, but the maximal inequality is a favorable indication for the research program aimed at showing their effectiveness.

Regarding seed-and-propagate on the hypercube we only add that seeding at a single vertex is insufficient, as the graph is sufficiently weakly connected that a sub-linear number of edge removals suffices to partition it into components each of sub-linear size. However multiple seedings remains a very viable strategy.

\subsection{Possible generalizations}
Let $G=(V,E)$ be any finite connected graph, with shortest-path metric
$d_G$. Let $G^{\square n}$ be the $n$th Cartesian power of $G$, the
graph on $V^n$ in which $(v,w)$ is an edge if there is a unique $i$
for which $(v_i,w_i) \in E$. The shortest path metric on $G^{\square
  n}$ is therefore the $L_1$ metric induced by $d_G$. Spherical
operators and spherical maximal operators are now defined, and we
conjecture that \thmref{main} holds for a suitable constant $A_G$.
Our proof techniques would require modification even for the next
simplest cases of $G$ taken to be the path or the complete graph on
three vertices.

In a different direction, the existence of a dimension-free bound for
all $\II^n$ begs the question whether there is a natural limit object
$T$, such that for every $\II^n$ there exists a single morphism $\II^n
\rightarrow T$ and thus each $n$ occurs in it as a special case.

\subsection{Proof overview}
Our proof is in two main steps, in each of which we obtain a
maximal inequality for one class of operators
based on comparison with another more tractable class. To
introduce the first of these reductions we need to define
the \textit{senate operator}\footnote{The terminology is to
  express that (as in the United States Senate) each block
  has equal weight without regard to size.}
 $\Sen$. Let $\cT=\{T_k\}$ be any
family of operators indexed by a parameter $k$
which varies over an interval $[0,a]$ ($a$ possibly
infinite) of either nonnegative reals or nonnegative
integers. (E.g., $\cS=\{S_k\}_0^n$ as above.) Then the family
$\Sen(\cT)=\{\Sen(\cT)_k\}$, indexed by $k$ in the same range,
consists of the operators
\[ \Sen(\cT)_k=\frac{1}{k+1}\sum_{\ell=0}^k T_\ell \quad \textrm{ or } \quad
\Sen(\cT)_k=\frac{1}{k}\int_{0}^k T_\ell \; d\ell \]
depending as $k$ ranges over integers or reals, and taking
the limit from above at $k=0$ in the continuous case.

In the first step of our argument we follow a comparison
method due to Stein~\cite{Stein61} to show that
\def\propSenS{
$\|M_{\cS}\|_{2 \to 2} \in O(1+\|M_{\Sen(\cS)}\|_{2 \to 2})$.}
\begin{prop}\label{prop:S-SenS}\label{PROP:S-SENS}
\propSenS
\end{prop}

Bounds on the Krawtchouk polynomials play a key role in this
argument. We will introduce the polynomials and prove these bounds in
\sect{krawt}, and then use them to prove \propref{S-SenS} in \sect{S-SenS}.

To introduce the second reduction we need to define the
family of stochastic \textit{noise operators} $\cN=\{N_t\}_{t
  \geq 0}$ indexed by real $t$. Letting $p=(1-e^{-t})/2$, we
set
\[ N_t=\sum_{k=0}^n \binom{n}{k} p^k(1-p)^{n-k} S_k \]
This has the following interpretation. $\pi_xN_tf$ is the
expectation of $\pi_y f$ where $y$ is obtained by running $n$
independent Poisson processes with parameter $1$ from time
$0$ to time $t$, and re-randomizing the $i$th bit of $x$ as many
times as there are events in the $i$th Poisson process. The
$N_t$'s form a semigroup: $N_{t_1}N_{t_2}=N_{t_1+t_2}$. The
process is equivalent to a Poisson-clocked random
walk on the hypercube.

We show (in \sect{senates}) by a direct pointwise comparison that:
\def\propsenates{
$\|M_{\Sen(\cS)}\|_{2 \to 2} \in O(\|M_{\Sen(\cN)}\|_{2 \to 2})$.}
\begin{prop}\label{prop:senates}\label{PROP:SENATES}
\propsenates
 \end{prop}

Finally (see~\cite{Krengel85}) $\|M_{\Sen(\cN)}\|_{2 \to 2}\leq 2\sqrt 2$
(indeed, $\|M_{\Sen(\cN)}\|_{p\to p} \leq (p/(p-1))^{1/p}$ for $p>1$)
by previous results: the Hopf-Kakutani-Yosida maximal
inequality and Marcinkiewicz's interpolation theorem.
For the reader's convenience, we restate these results
in the Appendix.

Combining these results, we have \thmref{main} by
$$ \|M_{\cS}f\|_2 \in O(\|f\|_2+\|M_{\Sen(\cS)}f\|_2) \subseteq
O(\|f\|_2+\|M_{\Sen(\cN)}f\|_2) \subseteq O(\|f\|_2).$$

\begin{remark}
While our main result is in terms of the $2\ra 2$ norm, many of our
techniques generalize to other norms.  Here we are
limited by \propref{S-SenS}, which does not conveniently generalize to
other norms. Very recently, this difficulty has been overcome by Krause; for a preliminary version of his work see~\cite{Krause-arxiv}.
On
the other hand, $\|M_{\cS}\|_{1\ra 1}=n+1$, as can be seen by taking
$f$ to be nonzero only on a single point.
\end{remark}

We are concerned in this paper solely with maximal operators
for sets $\cA$ of nonnegative matrices. For any such maximal
operator $M_\cA$, $|\pi_x M_\cA f| \leq \pi_x M_\cA |f|$. So
it suffices to show Theorem~\ref{thm:main} for nonnegative
$f$; this simplifies some expressions and will be assumed
throughout.

\section{Fourier analysis and Krawtchouk polynomials}
\label{sec:krawt}

For $y\in \II^n$, define the character $\chi_y\in \RR^{\II^n}$ by
$\pi_x\chi_y=(-1)^{x\cdot y} / \sqrt{2^n}$.    The
normalization is chosen so that the $\chi_y$ form an orthonormal
basis of $\RR^{\II^n}$. This basis simultaneously
diagonalizes each $S_k$, as they commute with $\II^n$ as an abelian group.  A
direct calculation (see also~\cite{FK02}) shows that
\be S_k \chi_y =
\kappa_k^{(n)}(|y|) \chi_y,
\label{eq:S_k-eig}\ee
where $|y|=|\{i:y_i =1\}|$, and $\kappa_k^{(n)}(|y|)$ is the
normalized $k^{\text{th}}$ Krawtchouk polynomial, defined by
\be \kappa_k^{(n)}(x) =
\sum_{j=0}^k (-1)^j \frac{\binom{x}{j}\binom{n-x}{k-j}}{\binom{n}{k}}
\label{eq:krawt-def}\ee

We collect here some facts about Krawtchouk polynomials.
\begin{lemma}\label{lem:krawt-prop}~
\begin{enumerate}
\item \label{k1} {\em $k$-$x$ Symmetry:} $\kappa_k^{(n)}(x)
  = \kappa_x^{(n)}(k)$.
\item \label{k2} {\em Reflection symmetry:}
  $\kappa_k^{(n)}(n-x) = (-1)^k \kappa_k^{(n)}(x)$.
\item \label{k3} {\em Orthogonality:}
\be\sum_{x=0}^{n} \kappa_k^{(n)}(x) \kappa_\ell^{(n)}(x) \binom{n}{x}
= \frac{2^n}{\binom{n}{k}} \delta_{k,\ell}.
\label{eq:krawt-ortho}\ee
\item \label{k4} {\em Roots:} The roots of
 $\kappa_k^{(n)}(x)$ are real, distinct, and lie
 in the range $n/2 \pm \sqrt{k(n-k)}$.
\end{enumerate}
\end{lemma}
The proofs of the first three claims are straightforward (see~\cite{Leven95}), and the fourth claim is a weaker version of Theorem~8 of~\cite{Krasi01}.
We sometimes abbreviate $\kappa_k(x) =\kappa_k^{(n)}(x)$.

Before going into further technical detail, we give an overview
of our goals in this section.  As we have noted in \sect{intro}, maximal
inequalities are easily proved for semigroups, such as the noise
operators $N_t$. In some ways $S_k$ resembles $N_{k/n}$, since
$N_{t}$ is approximately an average of $S_k$ for $k = nt \pm
\sqrt{nt(1-t)}$. While direct comparison is difficult (e.g. writing $S_k$
as a linear combination of $N_t$ necessarily entails large
coefficients), we can argue that the spectra of these operators should
be qualitatively similar.

Indeed, the $N_t$ are also diagonal in the $\chi_y$ basis, and for
$|y|=x$, their eigenvalue for $\chi_y$ is $(1-2t)^x$. Thus,
our goal in this section will be to show that $\kappa_k(x)$ has
similar behavior\footnote{This qualitative similarity breaks down when
$k$ (or $x$) is close to $n/2$, although this case will not be the
main contribution to our bounds.  See the introduction to~\cite{RegevK11}
for discussion of the properties of the $S_{n/2}$ operator.} to
$(1-2k/n)^x$.   More precisely, we prove

\begin{lemma}\label{lem:krawt-bound}
There is a constant $c>0$ such that for all $n$ and
integer $0\leq x,k \leq n/2$,
\be |\kappa_k^{(n)}(x)| \leq e^{-ckx/n}.
\label{eq:krawt-bound}\ee
\end{lemma}
Due to \lemref{krawt-prop}.\ref{k1},\ref{k2} it suffices to bound
$\kappa_k(x)$ only when $0 \leq k \leq x \leq n/2$.

\begin{proof}
The main complication in working with Krawtchouk polynomials is that
they have several different forms of asymptotic behavior depending on
whether $x$ and $k$ are in the lower, middle or upper part of their
range; indeed,~\cite{Dom08} breaks the asymptotic properties of
$\kappa_k(x)$ into 12 different cases.
However, for our purpose, we need only two different upper bounds on the Krawtchouk polynomials, based on whether $k/n$ is greater than or less than 0.14; a somewhat arbitrary threshold that we will see the justification for below.

\textbf{Case I: $k,x \geq 0.14n$.}

This is the simpler upper bound, which relies only on the orthogonality
property \lemref{krawt-prop}.\ref{k3}.  Setting $k=\ell$ in
\lemref{krawt-prop}.\ref{k3}
and observing that all of the terms on the LHS are nonnegative, it follows
that
\be \kappa_k^2(x) \leq \frac{2^n}{\binom{n}{k} \binom{n}{x}}.
\label{eq:ortho-bound}\ee

Based on Stirling's formula, Lemma 17.5.1 of~\cite{CT91} states that
\be \binom{n}{np}
\geq\frac{e^{nH_2(p)}}{\sqrt{8p(1-p)n}}
\geq\frac{e^{nH_2(p)}}{\sqrt{2n}}
\label{eq:stirling}\ee
 where $H_2(p):=-p\ln p - (1-p)\ln(1-p)$.

Note that $H_2(0.14) >\ln(2)/2$, so, for $k,x \geq 0.14n$, \eq{ortho-bound} implies that $\kappa_k^2(x) \leq 2n e^{(\ln(2)-2H_2(0.14))n}
\leq 2n e^{-c_1n}$
for some $c_1>0.116$. Now let $ n_0$ be sufficiently large ($n_0=100$ suffices) that
$ c_1 \geq 2\ln(2 n_0)/n_0$; then
for all $n \geq n_0$,
$2 n e^{-c_1n} \leq e^{-c_1n/2}$. So for all $n \geq n_0$
and all $0.14n \leq k,x \leq n/2$, $\kappa_k^2(x) \leq
 e^{-2c_1(n/2)(n/2)/n} \leq e^{-2c_1kx/n}$.

 To handle the $n<n_0$ case, we define
 \[ c_2=\min\{ -(n/kx)\ln |\kappa^{(n)}_k(x)| : 1 \leq k \leq x
 \leq n/2, n < n_0\}.\]
 It is immediate from
 Definition~\eq{krawt-def} that $|\kappa^{(n)}_k(x)|<1$ if
 $1 \leq k,x \leq n-1$, so $c_2>0$.

   Finally, the lemma
 follows with $c=\min\{2c_1,c_2\}$.

\textbf{Case II: $k \leq 0.14n$ or $x\leq 0.14n$.}

By the symmetry between $k$ and $x$, we can assume WLOG that $k\leq 0.14n$.

It is convenient to make the change of variable
\[ x = (1-z)n/2, \]
set
\[ \mu_k(z) =\kappa_k ((1-z)n/2), \]
and expand $\mu_k$ as $\mu_k(z) = \sum_{i=0}^k \alpha_{k,i} z^i$.
$\mu_k$ is either symmetric or anti-symmetric about $0$, and
we focus on bounding $|\mu_k|$ in the range $0\leq z\leq 1$
corresponding to $0 \leq x \leq n/2$.

Let $y_1,\ldots,y_k$ be the roots of $\mu_k$.  By \lemref{krawt-prop}.\ref{k2} the multiset $\{y_1,\ldots,y_k\}$ is identical to the multiset $\{-y_1,\ldots,-y_k\}$.
 So we can write
\be \mu_k^2(z) = \alpha_{k,k}^2
\prod_{i=1}^{k} (z^2-y_i^2).
\label{eq:mu_k-roots}\ee
It is immediate from Definition~\eq{krawt-def} that
\be \mu_k^2(1) = 1
\label{eq:mu_k-1}\ee

Furthermore, by \lemref{krawt-prop}.\ref{k4},
\be  y_{\max} := \max_i y_i
\leq\frac{2}{n}\sqrt{k(n-k)}
\label{eq:roots-bound}\ee

We now obtain an upper bound on $\mu_k^2(z)$
simply by maximizing \eq{mu_k-roots} subject to the constraints
\eq{mu_k-1} and \eq{roots-bound}. Observe that
$$\mu_k^2(z) = \frac{\mu_k^2(z)}{\mu_k^2(1)} =
 \prod_{i=1}^k
\frac{z^2-y_i^2}{1-y_i^2}.$$
Consider the problem of choosing $y_i$ to maximize a single term
$|z^2-y_i^2| / (1-y_i^2)$.
Observe that
\be \frac{\partial}{\partial y_i^2} \frac{z^2-y_i^2}{1-y_i^2} =
\frac{z^2-1}{(1-y_i^2)^2} \leq 0. \label{eq:deriv-neg}\ee
As a result, the maximum over $|y_i|< z$ is found at $y_i=0$ and the
maximum over $|y_i| > z$ is found at $y_i=y_{\max}$.  In the former
case, $|z^2-y_i^2|/(1-y_i^2)=z^2$. In the latter case,
$$ \frac{|z^2-y_i^2|}{1-y_i^2}
\leq \frac{y_{\max}^2-z^2}{1-y_{\max}^2}
\leq \frac{y_{\max}^2}{1-y_{\max}^2}
\leq 0.93.
$$
The last inequality uses the fact that $y_{\max} \leq
2\sqrt{0.14 \cdot 0.86}$ (recalling that $k\leq 0.14n$). So
$|z^2-y_i^2|/(1-y_i^2) \leq \max(z^2, 0.93)$,
implying that
\be \mu_k^2(z) \leq
 (\max(z^2, 0.93))^k \ee

If $z^2\geq 0.93$ then we use $z=1-2x/n \leq
e^{-2x/n}$ to obtain $\kappa_k^2(x) \leq e^{-4xk/n}$.

If $z^2\leq 0.93$ then (recalling $x \leq n/2$) we have
$\kappa_k^2(x) \leq (0.93)^k
\leq (0.93)^{2xk/n} = e^{-cxk/n}$ for $c=-2\ln(0.93)$.
\end{proof}

The threshold of 0.14 used here could be replaced by any $p$ satisfying $H_2(p)>1/2 > \sqrt{p(1-p)}$.

\section{Senates dominate dictatorships: proof of
Proposition~\ref{PROP:S-SENS}}\label{sec:S-SenS}
\begin{repprop}{prop:S-SenS}
\propSenS
\end{repprop}

We start by defining $\SB=\{S_k\}_0^{\lfloor n/2 \rfloor}$.   The
operator  $M_\SB:V  \to V$ is then defined by
\[ \pi_x M_\SB f =\max_{0 \leq k \leq \lfloor n/2 \rfloor} \pi_x S_k
f \]
Define $\iota := S_n$ to be the antipodal operator on
$\RR^{\II^n}$, i.e., the involution obtained by flipping all $n$ bits of the argument of the function. Then
$\pi_x M_{\cS}f = \max\{\pi_xM_\SB f,\pi_x\iota M_\SB f\}$. So
$\|M_{\cS}\|_{2 \to 2} \leq \sqrt{2}
\|M_\SB\|_{2 \to 2}$. Proposition~\ref{prop:S-SenS} therefore
follows from:

\begin{claim}\label{cl:S-SenS}
There is a $C<\infty$ such that for all $n$ and $f$,
$ \|M_\SB f\|_{2} \leq C\|f\|_2+\|M_{\Sen(\SB)}f\|_2$.
\end{claim}
We prove the claim in the following two subsections.

\subsection{A method of Stein}\label{subsec:stein61}
The bounds on $\|\SB_\ell\|_{2 \to 2}$ for even and odd
radius $\ell$ are technically distinct
(but not in any interesting way). We present the arguments in parallel.
\subsubsection{Even radius: $S_{2r}$ for
$0 \leq r \leq \rmx
=\lfloor\lfloor n/2 \rfloor /2 \rfloor$}
Note that for $n=4m+a$, $0 \leq a \leq 3$, this gives
$\rmx=m$.

Abel's lemma gives the following easily verified identity:
\be S_{2r} -\frac{1}{r+1}\sum_{k=0}^r S_{2k} =
\frac{1}{r+1}\sum_{k=1}^rk(S_{2k}-S_{2(k-1)})\ee

Hence we have the following pointwise (that is to say, valid
at each point $x$) inequality for $0 \leq r \leq \rmx$,
$\rmx=\lfloor\lfloor n/2 \rfloor /2 \rfloor$:

\ban
 [\pi_x(S_{2r}& -\frac{1}{r+1}\sum_{k=0}^r S_{2k})f]^2  = \\
& \big[\sum_{k=1}^r\frac{\sqrt{k}}{r+1} \cdot \sqrt{k}
[\pi_x (S_{2k}-S_{2(k-1)})f]\big]^2 \\
\leq& \sum_{k=1}^r\frac{k}{(r+1)^2}\sum_{k=1}^r
 k[\pi_x(S_{2k}-S_{2(k-1)})f]^2\\
 =&\frac{r}{2(r+1)}\sum_{k=1}^r
 k[\pi_x(S_{2k}-S_{2(k-1)})f]^2 \\
  \leq & \frac{1}{2}\sum_{k=1}^{\rmx}
 k[\pi_x(S_{2k}-S_{2(k-1)})f]^2
\ean
The first inequality is by Cauchy-Schwartz.
This suggests defining an ``error term'' operator $R_0:V \to  V$ by
\be \pi_x R_0f =
\sqrt{ \frac{1}{2}\sum_{k=1}^{\rmx}
 k[\pi_x (S_{2k}-S_{2(k-1)})f]^2 }
 \ee
so that for any $r\leq \rmx$,
\be \|(S_{2r}-\Sen(\cS)_{2r})f\|_2^2  \leq\|R_0f\|_2^2 \label{useR0} \ee

\subsubsection{Odd radius: $S_{2r+1}$ for $0 \leq r \leq \rmx=
\lfloor (\lfloor n/2 \rfloor - 1 )/2 \rfloor$}
Note that for $n=4m+a$, $0 \leq a \leq 3$, this gives
\[\rmx=\begin{cases} m-1 & \text{if } a \in \{0,1\} \\ m &
\text{if } a \in \{2,3\} \end{cases}.\]

Abel's lemma gives:
\be S_{2r+1} -\frac{1}{r+1}\sum_{k=0}^rS_{2k+1} =
\frac{1}{r+1}\sum_{k=1}^rk(S_{2k+1}-S_{2k-1})\ee

Hence we have the following pointwise inequality for $0 \leq r \leq
\rmx= \lfloor (\lfloor n/2 \rfloor - 1 )/2 \rfloor$:

\begin{multline*}
 [\pi_x \left(S_{2r+1} -\frac{1}{r+1}\sum_{k=0}^r S_{2k+1}\right)f]^2  =\\
 \big[\sum_{k=1}^r\frac{\sqrt{k}}{r+1} \cdot
 \sqrt{k}[\pi_x (S_{2k+1}-S_{2k-1})f]\big]^2 \\
\ldots  \leq  \frac{1}{2}\sum_{k=1}^{\rmx}
 k[\pi_x (S_{2k+1}-S_{2k-1})f]^2
\end{multline*}
This suggests defining an ``error term'' operator $R_1:V  \to  V$ by
\be \pi_x R_1f =
\sqrt{ \frac{1}{2}\sum_{k=1}^{\rmx}
 k[\pi_x (S_{2k+1}-S_{2k-1})f]^2}
 \ee
so that for any $r\leq \rmx$,
\be \|(S_{2r+1}-\Sen(\cS)_{2r+1})f\|_2^2  \leq\|R_1f\|_2^2
\label{useR1} \ee

\subsection{Bounding the error term}\label{subsec:krawtchouk_bounds}
Define $Rf$ by $\pi_x Rf :=\max\{ \pi_x R_0f, \pi_x R_1f\}$.
Combining (\ref{useR0}) and (\ref{useR1}) we have for each $x\in
\II^n$, each $f\in \RR^{\II^n}$ and each $r\leq n/2$,
$$|\pi_x (S_r-\Sen(\cS)_r)f| \leq \pi_x Rf.$$  Maximizing the LHS over $r$,
squaring and summing over $x$, we obtain that
$$\|M_{\SB}f-M_{\Sen(\SB)}f\|_2 \leq \|Rf\|_2.$$

Claim \ref{cl:S-SenS} (and Proposition~\ref{prop:S-SenS}) now follow from:
\begin{lemma} There is a $C<\infty$ such that $\|R_0f\|_2,\|R_1f\|_2 \leq
  C\|f\|_2$.
\end{lemma}
\begin{proof}
As seen in the preliminaries, the operators $S_k$ commute
and share the eigenvectors $\{\chi_y\}_{y \in \II^n}$, with
eigenvalues given by Eqn.~(\ref{eq:S_k-eig}): $S_k \chi_y =
\kappa_k^{(n)}(|y|) \chi_y$. Let $E_x$ be the projection
operator on the $\binom{n}{|y|}$-dimensional eigenspace
spanned by $\{\chi_y\}_{|y|=x}$; so $S_k=\sum_{x=0}^n
\kappa_k^{(n)}(x)E_x$.
We calculate:
\begin{eqnarray*}
\|R_0f\|_2^2 &=& \sum_{z \in \II^n}
\frac{1}{2}\sum_{k=1}^{\rmx}
 k[((S_{2k}-S_{2(k-1)})f)(z)]^2 \\
&=& \frac{1}{2}\sum_{k=1}^{\rmx} k
\sum_{x=0}^n
\|(\kappa_{2k}^{(n)}(x)-\kappa_{2(k-1)}^{(n)}(x))E_x f \|_2^2
\\
&=& \frac{1}{2}\sum_{x=0}^n \|E_x f \|_2^2
\sum_{k=1}^{\rmx}
k(\kappa_{2k}^{(n)}(x)-\kappa_{2(k-1)}^{(n)}(x))^2
\end{eqnarray*}
Likewise: (here and below the value of $\rmx$ depends on whether $R_0$ or $R_1$ is being bounded)
\begin{eqnarray*}
\|R_1f\|_2^2 &=& \sum_{z \in \II^n}
\frac{1}{2}\sum_{k=1}^{\rmx}
 k[((S_{2k+1}-S_{2k-1})f)(z)]^2 \\
&=& \frac{1}{2}\sum_{x=0}^n \|E_x f \|_2^2
\sum_{k=1}^{\rmx}
k(\kappa_{2k+1}^{(n)}(x)-\kappa_{2k-1}^{(n)}(x))^2
\end{eqnarray*}

Since $\|f\|_2^2=\sum_{x=0}^n \|E_x f \|_2^2$, it suffices
to show that there is a $C<\infty$ such
that for every $0 \leq x \leq n$,
 \begin{subequations} \label{eq:eval_diff}\ba
\left[\sum_{k=1}^{\rmx}
k\big(\kappa_{2k}^{(n)}(x)-\kappa_{2(k-1)}^{(n)}(x)\big)^2\right]&\leq C\\
 \left[\sum_{k=1}^{\rmx}
k\big(\kappa_{2k+1}^{(n)}(x)-\kappa_{2k-1}^{(n)}(x)\big)^2
 \right] &\leq C
\ea\end{subequations}
Recall that it suffices by \lemref{krawt-prop}.\ref{k2}
to consider $x \leq n/2$. For $x=0$,
(\ref{eq:eval_diff}) is trivial as the LHS
is $0$. For $x>0$ we use \lemref{krawt-prop}.\ref{k1} to
rewrite the parenthesized term (with $\ell=2k$ or
$\ell=2k+1$) as follows:

\be \nonumber
\kappa^{(n)}_{x}(\ell) -\kappa^{(n)}_{x}(\ell-1)  =
      -2\frac{{n-1\choose x-1}}{{n\choose
    x}}\kappa^{(n-1)}_{x-1}(\ell-1) =
      -\frac{2x}{n}\kappa^{(n-1)}_{x-1}(\ell-1)
\ee
To see this, recall that $\binom{n}{x}\kappa^{(n)}_x(\ell)$
counts $x$-subsets of $\{1,\ldots,n\}$ according to the
parity of their intersection with $\{1,\ldots,\ell\}$; now
condition on whether the $x$-subset contains the element
$\ell$.

Consequently,
\be \nonumber
\kappa^{(n)}_{x}(\ell) -\kappa^{(n)}_{x}(\ell-2)
 =
 -\frac{2x}{n}\big(\kappa^{(n-1)}_{x-1}(\ell-1) +
 \kappa^{(n-1)}_{x-1}(\ell-2)\big)
\ee
which by a similar argument is
\be
=-\frac{4x}{n} \frac{{n-2 \choose x-1}}{{n-1 \choose
    x-1}}\kappa^{(n-2)}_{x-1}(\ell-2)
=-\frac{4x}{n}\frac{(n-x)}{(n-1)}\kappa^{(n-2)}_{x-1}(\ell-2)
\ee
The two terms on the LHS of \eq{eval_diff} are now

\bea
\sum_{k=1}^{\rmx}k\big(\kappa^{(n)}_{x}(2k) -
  \kappa^{(n)}_{x}(2k-2)\big)^2
&=&
\sum_{k=1}^{\rmx}k\big(\frac{4x(n-x)}{n(n-1)}
  \kappa^{(n-2)}_{x-1}(2k-2)\big)^2
\sum_{k=1}^{\rmx}k\big(\kappa^{(n)}_{x}(2k+1) -
  \kappa^{(n)}_{x}(2k-1)\big)^2  \nonumber
\\&=&
\sum_{k=1}^{\rmx}k\big(\frac{4x(n-x)}{n(n-1)}
  \kappa^{(n-2)}_{x-1}(2k-1)\big)^2
\label{krawt-collapse}
\eea

For $x=1$,
quantities~(\ref{krawt-collapse})
come to $\sum_{k=1}^{\rmx}16kn^{-2}$ which is upper bounded
by a constant.

For $x>1$ we upper
bound~(\ref{krawt-collapse}), first
by
\[ \frac{16x^2}{n^2}\sum_{k=1}^{\rmx}
 k\big(\kappa^{(n-2)}_{x-1}(2k-2)\big)^2  \quad \text{and} \quad
 \frac{16x^2}{n^2}\sum_{k=1}^{\rmx}
 k\big(\kappa^{(n-2)}_{x-1}(2k-1)\big)^2
 \]
which in turn are upper bounded by (applying each value of
$\rmx$):
\[ \frac{16x^2}{n^2}\sum_{k=0}^{\lfloor n/2 \rfloor}
 (k/2+1)\big(\kappa^{(n-2)}_{x-1}(k)\big)^2 . \]
Now apply Lemma~\ref{lem:krawt-bound} to upper bound this by
\ba &\frac{16x^2}{n^2}\sum_{k=0}^{\infty}
 (k/2+1) e^{-2c(x-1)k/(n-2)} \nonumber\\
&= \frac{16x^2}{n^2}\sum_{k=0}^{\infty}
 (k/2+1) e^{-\alpha k} \label{eq:alpha-def}\\
&= \frac{x^2}{n^2}\frac{16(1-e^{-\alpha})+8e^{-\alpha}}{(1-e^{-\alpha})^2}
\label{eq:geom-id}\\
&\leq 24\left( \frac{x}{n (1-e^{-\alpha})} \right)^2 \nonumber\\
& \leq 24 \left( \frac{x}{n \alpha} \right)^2 \nonumber\\
& = 24 \left( \frac{(n-2)x}{2cn(x-1)} \right)^2 \nonumber\\
& \leq 24/c^2,\nonumber
\ea
where in \eq{alpha-def} we have defined $\alpha=2c(x-1)/(n-2)$ and in
\eq{geom-id} we have used the identity
$\sum_{k=0}^{\infty} ke^{-\alpha k} = e^{-\alpha} (1-e^{-\alpha})^{-2}$.

This completes the requirement of Eqn.~\ref{eq:eval_diff}.
\end{proof}

\section{Comparing senate maximal functions: proof of
Proposition~\ref{PROP:SENATES}}\label{sec:senates}

\begin{repprop}{prop:senates}\mbox{}

\propsenates
 \end{repprop}

\begin{proof}
The proof relies on pointwise comparison of maximal
functions.

If $A, B$ are matrices, write $A \leq B$ if $B-A$ is a
nonnegative matrix. If $\cA,\cB$ are sets of nonnegative
matrices indexed by integers or reals, we write $\cA
\lesssim \cB$ if for every $A \in \cA$ there is a
probability measure $\mu_A$ on $\cB$ such that $A \leq \int
B \; d\mu_A(B)$. Observe that in this case for any
nonnegative function $f$ and any $x$, $\sup_{A \in \cA}
\pi_x Af \leq \sup_{B \in \cB} \pi_xBf$, and therefore for any
norm, $\|M_\cA\| \leq \|M_\cB\|$ (and in particular for $\|
\cdot \|_{2 \to 2}$).

For any $k > \lfloor n/2 \rfloor $,
\[ \pi_x \Sen(\cS)_k f \leq \pi_x \Sen(\cS)_{\lfloor n/2 \rfloor}
(f+\iota f)
\]
Therefore $\|M_{\Sen(\cS)}\|_{2 \to 2} \leq
2\|M_{\Sen(\SB)}\|_{2 \to 2}$.
However, we will not compare $\Sen(\SB)$ and $\Sen(\cN)$ directly.
Instead, we will introduce a variant of $\cN$ that more closely
resembles $\SB$, but is no longer a semigroup.

Recall that $N_t$ represents the average over independently flipping each bit with
probability $p=(1-e^{-t})/2$.  Define $\tilde N_p$ to represent the
same noise process but parameterized by $p$ instead of $t$.  Thus
$$N_t
= \tilde N_{\frac{1-e^{-t}}{2}}
\qquad \text{and}\qquad
\tilde N_p = N_{-\ln(1-2p)}.$$
 While the sets $\{N_t\}_{t\geq 0}$ and
$\{\tilde N_p\}_{p\in [0,1/2)}$ are of course the same, their Senate
operators $\Sen(\cN)$ and $\Sen(\tilde\cN)$ are different:
\begin{align}
\Sen(\cN)_T
&= \frac{1}{T}\int_0^T N_t \; {\rm d}t
\label{eq:avg-over-t}\\
\Sen(\tilde \cN)_P & =
\frac{1}{P}\int_0^P \tilde N_p \; {\rm d}p
 =
\frac{1}{2P}\int_0^{-\ln(1-2P)} e^{-t} N_t \;{\rm d}t
\label{eq:avg-over-p}
\end{align}
In \eq{avg-over-p}, $P$ is varies over $(0,1/2)$ and in
\eq{avg-over-t}, $T$ can be any positive real number.

Hence \propref{senates} is established  in two subsidiary claims:
\begin{lemma} \label{lem:N-compare}
 $\Sen(\tilde \cN) \lesssim \Sen(\cN)$.
\end{lemma}

\begin{lemma} \label{lem:S-bar-N-tilde}
 $\Sen(\SB) \lesssim C\cdot\Sen(\tilde \cN)$ for some constant $C>0$.
\end{lemma}
\end{proof}

\begin{proof}[Proof of \lemref{N-compare}]
For each $P$ we will write $\Sen(\tilde\cN)_P$ as a convex combination
of $\Sen(\cN)_T$ for different values of $T$. By considering the
action of each side on the constant function we will see that it
suffices to write
$\Sen(\tilde\cN)_P$ as a bounded nonnegative combination
of $\Sen(\cN)_T$ for different values of $T$ (i.e. a linear
combination with coefficients that are nonnegative and whose sum is bounded).

Let $f(t) := e^{-t} / 2P$ and $\tau := -\ln(1-2P)$ so that
\eq{avg-over-p} becomes
\ba
\Sen(\tilde \cN)_P
&= \int_0^\tau {\rm d}t f(t) N_t  \\
&= \int_0^\tau {\rm d}t f(\tau) N_t + \int_0^\tau {\rm d}t
(f(t)-f(\tau)) N_t  \\
&= \tau f(\tau) \Sen(\cN)_\tau + \int_0^\tau {\rm d}t \left(
  \int_t^\tau {\rm d}T (-f'(T)) \right) N_t \\
&= \tau f(\tau) \Sen(\cN)_\tau + \int_0^\tau {\rm d}T \int_0^T {\rm d}t (-f'(T)) N_t \\
&= \tau f(\tau) \Sen(\cN)_\tau + \int_0^\tau {\rm d}T (-T f'(T))
\Sen(\cN)_T
\label{eq:nonnegative-combination}
\ea

Since $f'(T)<0$, \eq{nonnegative-combination} is the desired nonnegative
combination.  We have written the proof in this way to stress that the
only features of
\eq{avg-over-t} and \eq{avg-over-p} used are that $\Sen(\cN)_T$ is a
``flat'' distribution over $N_t$ and $\Sen(\tilde\cN)_P$ has $N_t$ weighted
by a weakly decreasing function.
\end{proof}

\begin{proof}[Proof of \lemref{S-bar-N-tilde}]
To compare $\Sen(\bar\cS)$ and $\Sen(\tilde \cN)$, we need to show that for any $K\leq
n/2$, we can find a distribution over $P$ such that $\Sen(\bar\cS)_K$ is pointwise $\leq$ the appropriate average over $\Sen(\tilde\cN)_P$ times a constant.
In fact, it will suffice to consider a distribution that is concentrated on a single value of $P$.
   Define $P_K:= \min(\frac{K+\sqrt{K}}{n}, \frac{1}{2})$.  In \lemref{binom-LB}, we
will show that $\Sen(S)_K \leq C\cdot \Sen(\tilde\cN)_{P_K}$,
thus implying that $\Sen(\bar\cS) \lesssim
C\cdot\Sen(\tilde \cN)$.  The idea behind \lemref{binom-LB} is that
for each $k\leq K$, there are significant contributions to the $S_k$
coefficient of $\Sen(\tilde\cN)_{P_K}$ for $p$ throughout the range
$[k/n,(k+\sqrt k)/n]$.  This window has width $\sqrt{k}/n$,
contributes $\Omega(1/\sqrt k)$ weight to $S_k$ at each point and is
normalized by $\frac{1}{P_K}\approx \frac{n}{K}$.  The total
contribution is thus $\Omega(1/K)$.

\begin{lemma}\label{lem:binom-LB}
Let $n\geq 9$, $K\leq n/2$ and $P_K = \min(\frac{K+\sqrt{K}}{n},
\frac{1}{2})$.
Then $\Sen(S)_K \leq 3e^{20}\cdot \Sen(\tilde\cN)_{P_K}$,
\end{lemma}
The true constant is certainly much better, and perhaps closer to
$1/2$.  We remark that $n\geq 9$ can be assumed WLOG, since
$\|M_{\cS}\|_{2\ra 2} \leq n+1$ by the triangle inequality.

\begin{proof}
Observe that $P:=P_K\leq 2K/n$.

For $0\leq k\leq K$, we now compare the coefficient of $S_k$ in
$\Sen(\cS)_K$ (where it is $1/(K+1)$) to its value in
$\Sen(\tilde\cN)_P$, where it is $\frac{1}{P}\int_0^P {\rm
  d}p\,B(n,p,k)$.  Denote this latter quantity by $a_k$.

Consider first $k=0$.  If $K=0$ and $P=0$, then $S_0$ has weight 1 in
both cases.  Otherwise $P\geq 1/n$, and
$$a_0 \geq \frac{1}{P}\int_0^{1/n} (1-p)^n {\rm d}p
\geq \frac{n}{2K}\cdot \frac{1}{n}\cdot \left(1-\frac{1}{n}\right)^n
\geq \frac{1}{8K}.$$
This last inequality uses the fact that $(1-1/n)^n$ is increasing with
$n$ and thus is $\geq 1/4$ for $n\geq 2$.

Next, suppose $k+\sqrt{k} \leq \frac{n}{2}$.  Then $P \geq
\frac{k + \sqrt{k}}{n}$.  We will consider only the contribution to
$a_k$ resulting from $0 \leq p-k/n \leq \sqrt{k}/n$.  First, use
Stirling's formula
to bound
$$\binom{n}{k} \geq \exp(n H_2(k/n)) \sqrt{\frac{n}{9 k(n-k)}}
\geq \exp(n H_2(k/n)) \frac{1}{3\sqrt{k}},$$
implying that
\be B(n,\frac{n}{k},k) \geq \frac{1}{3\sqrt{k}} .
\label{eq:binom-peak}\ee
Next, observe that
\be \frac{{\rm d}}{{\rm d}q} \ln B(n,q,k) =
\frac{k-nq}{q(1-q)}.\label{eq:log-deriv}\ee
When
$k/n \leq q \leq p \leq (k+\sqrt{k})/n$, we have $k-nq \geq k-np \geq -\sqrt{k}$ and
$q(1-q) \geq \frac{k}{n}(1-\frac{k}{n}) \geq \frac{k}{2n}$.  Thus,
$0 \leq p-k/n \leq \sqrt{k}/n$ implies that
\begin{subequations}\label{eq:binom-ratio}
\ba \ln\frac{B(n,p,k)}{B(n,\frac{k}{n},k)}
&= \int_{k/n}^p {\rm d}q \frac{k-nq}{q(1-q)}
\\&\geq \left(p-\frac{k}{n}\right) \frac{k-np}{k/2n}
\\&\geq -2 \ea
\end{subequations}
Combining \eq{binom-peak} and \eq{binom-ratio}, we find that
$B(n,p,k)\geq 1/6e^2\sqrt{k}$ for $k/n \leq p \leq (k+\sqrt{k})/n$.
Thus
$$a_k \geq \frac{1}{P}\int_{k/n}^{(k+\sqrt k)/n} {\rm d}p\, B(n,p,k)
\geq \frac{n}{2K}\cdot\frac{\sqrt k}{n}\cdot\frac{1}{6e^2\sqrt k}
\geq \frac{1}{12e^2K}.$$

Finally, we consider the case when $k+\sqrt{k} \geq n/2$.  In this
regime, we have $P=1/2$.  Also, $k\leq n/2$, so $K\geq k\geq
n/2-\sqrt{n/2}\geq n/4$ (assuming in the last step that $n\geq 8$).

Consider $p\in [1/2-1/\sqrt{n}, 1/2]$.
Assume that $n\geq 9$ so that $1/2-1/\sqrt{n}\geq 1/6$.  We can then
use \eq{log-deriv} to bound
\ban \ln\frac{B(n,\frac{1}{2}-\frac{1}{\sqrt{n}},k)}{B(n,\frac{k}{n},k)}
&\geq -\left(\frac{k}{n} -\frac{1}{2} + \frac{1}{\sqrt{n}}\right)
\frac{k-\frac{n}{2}+\sqrt{n}}{\frac{1}{6}\cdot\frac{1}{2}}\\
&\geq -12\left(1-\frac{1}{\sqrt{2}}\right)^2 \geq -20.\ean
Similarly
$$\ln\frac{B(n,\frac{1}{2},k)}{B(n,\frac{k}{n},k)} \geq
\left(\frac{1}{2}-\frac{k}{n}\right)\frac{k-n/2}{1/12} \geq -6.$$
We conclude that
$$a_k \geq
\frac{1}{P}\int_{\frac{1}{2}-\frac{1}{\sqrt{n}}}^{\frac{1}{2}}
{\rm d}p\, B(n,p,k)
\geq 2 \cdot \frac{1}{\sqrt{n}}\cdot \frac{1}{3e^{20}\sqrt{k}}
\geq \frac{1}{3e^{20}K}. $$
\end{proof}
This completes the proof of \lemref{S-bar-N-tilde}.
\end{proof}

\appendix
\section{Maximal inequalities on semigroups}
In this appendix, we review the maximal ergodic inequality that we use
to show $\|M_{\Sen(\cN)}\|_{1\ra 1,w}\leq 1$.

\begin{theorem}[maximal ergodic inequality]\label{thm:ergodic}
 Let $A$ be a positive sublinear operator with
  $\|A\|_{1\ra 1}\leq 1$ and $\|A\|_{\infty\ra\infty}\leq 1$.  Let
  ${\cal A}$ denote the (discrete) semigroup generated by $A$.
Then
$$\|M_{\Sen({\cal A})}\|_{1\ra 1,w} \leq 1.$$
\end{theorem}

This inequality was originally proved by Kakutani and Yosida (in the case
when $A$ is linear) and Hopf (for general semigroups).  However since
the proof of~\cite{Garsia65} is simple and self-contained, we restate
it here.

\begin{proof}
For $f$ an arbitrary function and $T$ a nonnegative integer, define $E^T_f =
I[M_{\Sen({\cal A})_{\leq T}} f \geq 0]$.  We will first prove that
\be \langle E^T_f, f\rangle \geq 0,
\label{eq:KYH-ineq}\ee
for all $T$.  To see that the theorem follows from this claim, apply
\eq{KYH-ineq} to $f-\lambda$ and we find that $\lambda
\|E^T_{f-\lambda}\|_1 \leq \langle E^T_{f-\lambda}, f\rangle
\leq \|f\|_1$.  Thus $\|E^T_{f-\lambda}\|_1 \leq
\lambda^{-1}\|f\|_1$.  Since this inequality holds for all $T$, it
implies that $\|I[M_{\Sen({\cal A})}f\geq \lambda]\|_1$ is also $\leq
\lambda^{-1}\|f\|_1$.

We now return to the proof of \eq{KYH-ineq}.  We define for this
purpose an unweighted Senate operator:
$$ \overline{\Sen}({\cal A})_T
 = \sum_{t=0}^T A_t.$$
We abbreviate $M_T := M_{\Sen(\cA)_{\leq T}}$ and
 $\bar M_T := M_{\overline{\Sen}(\cA)_{\leq T}}$.
Observe that $I[M_Tf \geq 0] = I[\bar M_Tf \geq 0]$.
Also define, for any function $g$, the function $(g)^+$ to be the
nonnegative part of $g$.

Thus, if $0\leq t\leq T$, then we have $\overline{\Sen}(\cA)_tf \leq
\bar M_Tf \leq (\bar M_Tf)^+$.
This implies that
$$f + A(\bar M_Tf)^+ \geq f + A \overline{\Sen}(\cA)_tf
 = \overline{\Sen}(\cA)_{t+1}f.$$
Thus, $f \geq \overline{\Sen}(\cA)_t f - A((\bar M_Tf)^+),$ (including
an $t=0$ case that can be checked separately) and
maximizing over $0\leq t\leq T$, we have
$$f \geq \bar M_Tf - A((\bar M_Tf)^+).$$
Now we take the inner product of both sides with $E^T_f$ and find
\ba
\langle E^T_f, f \rangle & \geq
\langle E^T_f, \bar M_Tf - A((\bar M_Tf)^+) \rangle \\
& = \langle E^T_f, (\bar M_Tf)^+ - A((\bar M_Tf)^+) \rangle \\
& = \|(\bar M_Tf)^+\|_1 - \langle E^T_f, A((\bar M_Tf)^+) \rangle \\
& \geq  \|(\bar M_Tf)^+\|_1 - \|A((\bar M_Tf)^+)\|_1 \\
& \geq 0 \label{eq:because-contractive}
\ea
The final inequality uses the fact that $\|A\|_{1\ra 1}\leq 1$.
\end{proof}

For our purposes, we will want to convert the bound on the $1\ra 1,w$
norm into a bound on $p\ra p$ norms, especially for $p=2$.  This is
achieved by the Marcinkiewicz interpolation theorem~\cite{Zygmund56}.

\begin{theorem}
Let $A$ be a sublinear operator with $\|A\|_{p\ra p,w} \leq N_p$ and
$\|A\|_{q\ra q,w} \leq N_q$.  Then for any $1\leq p<r<q \leq \infty$, we have
$$\|A\|_{r\ra r} \leq 2 N_p^\delta N_q^{1-\delta}
 \left(\frac{r(q-p)}{(r-p)(q-r)}\right)^{1/r},$$
 with $\delta = p(q-r)/r(q-p)$.
\end{theorem}

In our case, a maximal operator $M_\cA$ has $\|M_\cA\|_{\infty\ra
  \infty,w}=1$ and if $\cA$ is a positive contractive semigroup, then
\thmref{ergodic} implies that $\|M_\cA\|_{1\ra 1,w} \leq 1$ as well.
This implies that for any $1<p$, we have
\be \|M_{\cA}\|_{p\ra p} \leq
2\left(\frac{p}{p-1}\right)^{1/p},
\ee
which is $2\sqrt{2}$ when $p=2$.

\section*{Acknowledgments} Thanks to Gil Kalai for a stimulating
conversation which pointed us in the direction of maximal
inequalities, to Kostantin Makarychev and Yury Makarychev for discussions about
UGC on the hypercube, to Yuval Peres for helpful comments, and to Terence Tao for suggestions along the lines of
\sect{S-SenS}.
AWH was funded by NSF grants 0916400, 0829937 and 0803478 and DARPA
QuEST
contract FA9550-09-1-0044. AK was at Microsoft Research at the beginning and during part of this work. LJS was funded in part by NSF grants CCF-0829909, CCF-1038578, CCF-1319745, and the NSF-supported Institute for Quantum Information and Matter; this work began during his visit in 2010 to the Theory Group at Microsoft Research, Redmond.

\bibliographystyle{abbrv}
\bibliography{toc/refs-final}
\end{document}